\documentclass[11pt,twoside,reqno]{amsart}

\usepackage{version}         

 \usepackage{color}

\usepackage{multirow}
\usepackage[OT1]{fontenc}
\usepackage{type1cm}
\usepackage{amssymb}
\usepackage{mathrsfs}

\usepackage[left=2.3cm,top=3cm,right=2.3cm]{geometry}
\numberwithin{equation}{section}

\theoremstyle{plain}
\newtheorem{theorem}{Theorem}[section]

\newtheorem{corollary}[theorem]{Corollary}
\newtheorem{proposition}[theorem]{Proposition}
\newtheorem{lemma}[theorem]{Lemma}
\newtheorem{notation}[theorem]{Notation}

\theoremstyle{remark}
\newtheorem{remark}[theorem]{Remark}

\newtheorem{example}[theorem]{Example}
\newtheorem*{ack}{Acknowledgement}

\theoremstyle{definition}
\newtheorem{definition}[theorem]{Definition}

\newcommand{\R}{\mathbb{R}}

\newcommand{\N}{\mathbb{N}}

\begin{document}

\title[Distributions of full and non-full words in beta-expansions]{Distributions of full and non-full words in beta-expansions}

\author{Yao-Qiang Li}
\address{Department of Mathematics \\
         South China University of Technology \\
         Guangzhou, 510641 \\
         P.R.\ China}
\email{scutyaoqiangli@qq.com}

\author{Bing Li$^*$}
\address{Department of Mathematics \\
         South China University of Technology \\
         Guangzhou, 510641 \\
         P.R.\ China}
\email{scbingli@scut.edu.cn}

\thanks{*Corresponding author.}
\date{\today}

\begin{abstract}
The structures of full words and non-full for $\beta$-expansions are completely characterized in this paper. We obtain the precise lengths of all the maximal runs of full and non-full words among admissible words with same order.
\end{abstract}

\maketitle

\section{Introduction}

Let $\beta>1$ be a real number. The $\beta$-expansion was introduced by R\'enyi \cite{Ren57} in 1957, which generalized the usual decimal expansions (generally $N$-adic expansion with integers $N>1$) to that with any real base $\beta$. There are some different behaviors for the representations of real numbers and corresponding dynamics for the integer and noninteger cases. For example, when $\beta\in\N$, every element in $\{0,1,\cdots,\beta-1\}^{\N}$ (except countablely many ones) is the $\beta$-expansion of some $x\in[0,1)$ (called admissible sequence). However, if $\beta\notin\N$, not any sequence in $\{0,1,\cdots,\lfloor\beta\rfloor\}^{\N}$ is the $\beta$-expansion of some $x\in[0,1)$ where $\lfloor\beta\rfloor$ denotes the integer part of $\beta$. Parry \cite{Pa60} managed to provide a criterion for admissability of sequences (see Lemma \ref{charADM} below). Any finite truncation of an admissible sequence is called an admissible word. Denoted by $\Sigma_\beta^n$ the set of all admissible words with length $n\in\N$. By estimating the cardinality of $\Sigma_\beta^n$ in \cite{Ren57}, it is known that the topological entropy of $\beta$-transformation $T_\beta$ is $\log\beta$. The projection of any word in $\Sigma_\beta^n$ is a cylinder of order $n$ (also say a fundamental interval), which is a left-closed and right-open interval in $[0,1)$. The lengths of cylinders are irregular for $\beta\notin\N$, meanwhile, they are all regular for $\beta\in\N$, namely, the length of any cylinder of order $n$ equals $\beta^{-n}$. Li and Wu \cite{LiWu08} introduced a classification of $\beta>1$ for characterising the regularity of the lengths of cylinders and then the sizes of all corresponding classes were given by Li, Persson, Wang and Wu \cite{LPWW14} in the sense of measure and dimension. Another different classification of $\beta>1$ was provided by Blanchard \cite{Bla89} from the viewpoint of dynamical system, and then the sizes of all corresponding classes were given by Schmeling \cite{Schme97} in the sense of topology, measure and dimension.

A cylinder with order $n$ is said to be full if it is mapped by the $n$-th iteration of $\beta$-transformation $T_\beta^n$ onto $[0,1)$ (see Definition \ref{Deffull} below, \cite{Wal78} or \cite{DK02}) or equivalently its length is maximal, that is, equal to $\beta^{-n}$ (see Proposition \ref{charFULL} below, \cite{FW12} or \cite{BuWa14}). An admissible word is said to be full if the corresponding cylinder is full. Full words and cylinders have very good properties. For example, Walters \cite{Wal78} proved that for any given $N > 0$, $[0,1)$ is covered by the full cylinders of order at least $N$. Fan and Wang \cite{FW12} obtained some good properties of full cylinders (see Proposition \ref{charFULL} and Proposition \ref{fullzero} below). Bugeaud and Wang \cite{BuWa14} studied the distribution of full cylinders, showed that for $n\ge1$, among every $(n + 1)$ consecutive cylinders of order $n$, there exists at least one full cylinder, and used it to prove a modified mass distribution principle to estimate the Hausdorff dimension of sets defined in terms of $\beta$-expansions. Zheng, Wu and Li  proved that the extremely irregular set is residual with the help of the full cylinders (for details see \cite{ZWL17}).

In this paper, we are interested in the distributions of full and non-full words in $\Sigma_\beta^n$, i.e., the distributions of full and non-full cylinders in $[0,1)$. More precisely, we consider the lexicographically ordered sequence of all order $n$ admissible words, and count the numbers of successive full words and successive non-full words. Or, in what amounts to the same thing, we look at all the fundamental intervals of order $n$, arranged in increasing order along the unit interval, and ask about numbers of successive intervals where $T_\beta^n$ is onto (and numbers of intervals where it is not onto). Our main results concern the maximal number of successive full words, and the maximal number of successive non-full words as a function of $n$ and $\beta$. In particular, the dependence on $\beta$ is expressed in terms of the expansion of 1 with base $\beta$.

The main objective of this paper is to describe the structure of admissible words and the precise lengths of the maximal runs of full words and non-full words (see Definition \ref{maximal}). The concept of maximal runs is a new way to study the distribution of full words and cylinders.
 Firstly Theorem \ref{strADM} gives a unique and clear form of any admissible word, and Theorem \ref{strFULL} and Corollary \ref{tail-non-full} provide some convenient ways to check whether an admissible word is full or not.
  Secondly Theorem \ref{fulllength} describes all the precise lengths of the maximal runs of full words, which indicates that such lengths rely on the nonzero terms in the $\beta$-expansion of $1$. Consequently, the maximal and minimal lengths of the maximal runs of full words are given in Corollary \ref{maxlengthfull} and Corollary \ref{minlengthfull} respectively.  Finally by introducing a function $\tau_\beta$ in Definition \ref{tau}, a similar concept of numeration system and greedy algorithm, we obtain a convenient way to count the consecutive non-full words in Lemma \ref{tau-non-full}, which can easily give the maximal length of the runs of non-full words in Corollary \ref{maxlengthnotfull}  and generalize the result of Bugeaud and Wang mentioned above (see Remark \ref{n+1}).
 Furthermore, all the precise lengths of the maximal runs of non-full words are stated in Theorem \ref{notfulllength}, which depends on the positions of nonzero terms in the $\beta$-expansion of $1$. Moreover, the minimal lengths of the maximal runs of non-full words are obtained in Corollary \ref{minlengthnotfull}.

This paper is organized as follows. In Section 2, we introduce some basic notation and preliminary work needed. In Section 3, we study the structures of admissible words, full words and non-full words as basic results of this paper. In Section 4 and Section 5, we obtain all the precise lengths of the maximal runs of full words and non-full words respectively as the main results.

\section{Notation and preliminaries}

Let us introduce some basic notation and preliminary work needed. Let $\beta>1$.

$\bullet$ Let $T_\beta:[0,1)\rightarrow[0,1)$ be the map:
$$T_\beta(x):=\beta x-\lfloor\beta x \rfloor,\quad x\in[0,1).$$
Let $\mathcal{A}_\beta=\{0,1,\cdots,\beta-1\}$ when $\beta\in\N$, $\mathcal{A}_\beta=\{0,1,\cdots,\lfloor\beta\rfloor\}$ when $\beta\notin\N$ and
$$\epsilon_n(x,\beta) := \lfloor \beta T^{n-1}_\beta (x)\rfloor,\quad n\in\N, x\in[0,1).$$
Then $\epsilon_n(x,\beta) \in \mathcal{A}_\beta$ and
$$x = \sum_{n = 1}^\infty \epsilon_n(x,\beta)\beta^{-n}.$$
The sequence $\epsilon(x,\beta):=\epsilon_1(x,\beta)\epsilon_2(x,\beta)\cdots\epsilon_n(x,\beta)\cdots$
is also called the $\beta$-\textit{expansion} of $x$. The system $([0,1),T_\beta)$ is called a $\beta$-\textit{dynamical system}.

$\bullet$ Define
$$T_\beta(1):=\beta-\lfloor\beta\rfloor\mbox{ and }\epsilon_n(1,\beta) := \lfloor \beta T^{n-1}_\beta (1)\rfloor,\quad n\in\N.$$
Then the number $1$ can also be expanded into a series, denoted by
$$1 = \sum_{n = 1}^\infty \epsilon_n(1,\beta)\beta^{-n}.$$
The sequence $\epsilon(1,\beta):=\epsilon_1(1,\beta)\epsilon_2(1,\beta)\cdots\epsilon_n(1,\beta)\cdots$ is also called the $\beta$-\textit{expansion} of $1$.
For simplicity, we write $\epsilon(1,\beta) = \epsilon_1\epsilon_2\cdots\epsilon_n\cdots.$

$\bullet$
If there are infinitely many $n$ with $\epsilon_n\neq 0$, we say that $\epsilon(1,\beta)$ is infinite. Otherwise, there exists $M\in \mathbb{N}$ such that $\epsilon_M\neq0$ with $\epsilon_j=0$ for all $j>M$,  $\epsilon(1,\beta)$ is said to be finite, sometimes say that $\epsilon(1,\beta)$ is finite with length $M$.
The \textit{modified} $\beta$-expansion of $1$ is defined as
$$\epsilon^*(1,\beta) := \epsilon(1,\beta)$$
if $\epsilon(1,\beta)$ is infinite, and
$$\epsilon^*(1,\beta) := (\epsilon_1\cdots\epsilon_{M-1}(\epsilon_M-1))^\infty$$
if $\epsilon(1,\beta)$ is finite with length $M$. Here for a finite word $w \in \mathcal{A}_\beta^n$, the periodic sequence $w^\infty \in \mathcal{A}_\beta^\N$ means that
$$w^\infty:= w_1w_2\cdots w_n w_1w_2\cdots w_n \cdots.$$
In this paper, we always denote
$$\epsilon^*(1,\beta)=\epsilon_1^*\epsilon_2^*\cdots\epsilon_n^*\cdots$$
no matter whether $\epsilon(1,\beta)$ is finite or not.

$\bullet$ Let $\prec$ and $\preceq$ be the \textit{lexicographic order} in $\mathcal{A}_\beta^\N$. More precisely, $w\prec w'$ means that there exists $k\in\mathbb{N}$ such that $w_i = w_i'$ for all $1\le i < k$ and $w_k < w_k'$. Besides, $w\preceq w'$ means that $w\prec w'$ or $w=w'.$ Similarly, the definitions of $\prec$ and $\preceq$ are extended to the sequences by identifying a finite word $w$ with the sequence $w0^\infty$.

$\bullet$ For any $w\in \mathcal{A}_\beta^\N$, we use $w|_k$ to denote the prefix of $w$ with length $k$, i.e., $w_1w_2\cdots w_k$ where $k\in\N$. For any $w\in \mathcal{A}_\beta^n$, we use $|w|:=n$ to denote the \textit{length} of $w$ and $w|_k$ to denote the prefix of $w$ with length $k$ where $1\le k\le |w|$.

$\bullet$ Let $\sigma:\mathcal{A}_\beta^\N\rightarrow\mathcal{A}_\beta^\N$ be the \textit{shift}
$$\sigma(w_1w_2\cdots) = w_2w_3\cdots \quad \text{for } w \in \mathcal{A}_\beta^\N$$
and $\pi_\beta:\mathcal{A}_\beta^\N\rightarrow\R$ be the projection map
$$\pi_\beta(w)=\frac{w_1}{\beta}+\frac{w_2}{\beta^2}+\cdots+\frac{w_n}{\beta^n}+\cdots \quad \text{for } w \in \mathcal{A}_\beta^\N.$$

\begin{definition}[Admissability]\indent
\begin{itemize}
\item[(1)] A word $w \in \mathcal{A}_\beta^n$ is called \textit{admissible}, if there exists $x \in [0,1)$ such that $\epsilon_i(x,\beta) = w_i$ for $i = 1,\cdots,n$. Denote
$$\Sigma_\beta^n:=\{w \in \mathcal{A}_\beta^n: w \text{ is admissible}\} \text{ and } \Sigma_\beta^*:=\bigcup_{n=1}^\infty\Sigma_\beta^n.$$
\item[(2)] A sequence $w \in \mathcal{A}_\beta^\N$ is called \textit{admissible}, if there exists $x \in [0,1)$ such that $\epsilon_i(x,\beta) = w_i$ for all $i \in \N$. Denote
$$\Sigma_\beta:=\{w \in \mathcal{A}_\beta^\N: w \text{ is admissible}\}.$$
\end{itemize}
\end{definition}

Obviously, if $w\in\Sigma_\beta$, then $w|_n\in\Sigma_\beta^n$ and $w_{n+1}w_{n+2}\cdots\in\Sigma_\beta$ for any $n\in\mathbb{N}$. By the algorithm of $T_\beta$, it is easy to get the following lemma.

%

\begin{lemma}\label{1cutADM}
For any $n \in \mathbb{N}$, $\epsilon^*(1,\beta)|_n \in\Sigma_\beta^n$ and is maximal in $\Sigma_\beta^n$ with lexicographic order .
\end{lemma}

The following criterion for admissible sequence is due to Parry.
\begin{lemma}[\cite{Pa60}]\label{charADM}
Let $w \in \mathcal{A}_\beta^\N$.
Then $w$ is admissible (that is, $w \in \Sigma_\beta$) if and only if
$$\sigma^k (w) \prec \epsilon^*(1,\beta) \quad \text{for all } k \ge 0.$$
\end{lemma}

As a corollary of Parry's criterion, the following lemma can be found in \cite{Pa60}.
\begin{lemma}\label{1exp}Let $w$ be a sequence of non-negative integers. Then $w$ is the $\beta$-\textit{expansion} of $1$ for some $\beta>1$ if and only if $\sigma^kw\prec w$ for all $k\ge 1$. Moreover, such $\beta$ satisfies $w_1\le\beta<w_1+1$.
\end{lemma}

\begin{definition}[cylinder]\label{Cylinder}
Let $w\in\Sigma_\beta^*$. We call
$$[w]:=\{v\in\Sigma_\beta:v_i=w_i \text{ for all }1\le i\le |w|\}$$
the \textit{cylinder} generated by $w$ and
$$I(w):=\pi_\beta([w])$$
the \textit{cylinder} in [0,1) generated by $w$.
\end{definition}

\begin{definition}[full words and cylinders]\label{Deffull}
Let $w\in\Sigma_\beta^n$. If $T_\beta^nI(w)=[0, 1)$, we call the word $w$ and the cylinders $[w],I(w)$ \textit{full}. Otherwise, we call them \textit{non-full}.
\end{definition}

\begin{lemma}[\cite{LiWu08}, \cite{FW12}, \cite{BuWa14}]\label{admfull}
Suppose the word $w_1\cdots w_n$ is admissible and $w_n \neq 0$. Then $w_1 \cdots w_{n-1} w_n'$ is full for any $w_n' < w_n$.
\end{lemma}

\section{The structures of admissible words, full words and non-full words}

The following proposition is a criterion of full words. The equivalence of (1), (2) and (4) can be found in \cite{FW12}. We give some proofs for self-contained and more characterizations (3), (5), (6) are given here.
\begin{proposition}\label{charFULL}
Let $w \in \Sigma_\beta^n$. Then the following are equivalent.
\begin{itemize}
\item[\emph{(1)}] $w$ is full, i.e., $T_\beta^nI(w)=[0, 1)$;
\item[\emph{(2)}] $|I(w)| = \beta^{-n}$;
\item[\emph{(3)}] The sequence $ww'$ is admissible for any $w' \in \Sigma_\beta$;
\item[\emph{(4)}] The word $ww'$ is admissible for any $w' \in \Sigma_\beta^*$;
\item[\emph{(5)}] The word $w\epsilon_1^*\cdots\epsilon_k^*$ is admissible for any $k\ge 1$;
\item[\emph{(6)}] $\sigma^n[w]=\Sigma_\beta$.
\end{itemize}
\end{proposition}
\begin{proof}$(1)\Rightarrow(2)$ Since $w$ is full, $T_\beta^nI(w)=[0,1)$. Noting that
$$x=\frac{w_1}{\beta}+\cdots+\frac{w_n}{\beta^n}+\frac{T_\beta^nx}{\beta^n}\quad\text{for any $x\in I(w)$},$$
we can get
$$I(w)=[\frac{w_1}{\beta}+\cdots+\frac{w_n}{\beta^n},\frac{w_1}{\beta}+\cdots+\frac{w_n}{\beta^n}+\frac{1}{\beta^n}).$$
 Therefore $|I(w)| = \beta^{-n}$.
\newline$(2)\Rightarrow(3)$ Let $x,x'\in[0,1)$ such that $\epsilon(x,\beta)=w0^\infty$ and $ \epsilon(x',\beta)=w'$. Then
$$x=\frac{w_1}{\beta}+\cdots+\frac{w_n}{\beta^n}\quad\text{and}\quad x'=\frac{w_1'}{\beta}+\frac{w_2'}{\beta^2}+\cdots.$$
Let
$$y=x+\frac{x'}{\beta^n}=\frac{w_1}{\beta}+\cdots+\frac{w_n}{\beta^n}+\frac{w_1'}{\beta^{n+1}}+\frac{w_2'}{\beta^{n+2}}\cdots.$$
We need to prove $ww'\in\Sigma_\beta$. It suffices to prove $y\in[0,1)$ and $\epsilon(y,\beta)=ww'$. In fact, since $I(w)$ is a left-closed and
right-open interval with $\frac{w_1}{\beta}+\cdots+\frac{w_n}{\beta^n}$ as its left endpoint and $|I(w)| = \beta^{-n}$, we get
$$I(w)=[\frac{w_1}{\beta}+\cdots+\frac{w_n}{\beta^n},\frac{w_1}{\beta}+\cdots+\frac{w_n}{\beta^n}+\frac{1}{\beta^n})=[x,x+\frac{1}{\beta^n}).$$
So $y\in I(w)\subset[0,1)$ and $\epsilon_1(y,\beta)=w_1,\cdots,\epsilon_n(y,\beta)=w_n$. That is
$$y=\frac{w_1}{\beta}+\cdots+\frac{w_n}{\beta^n}+\frac{T_\beta^ny}{\beta^n}=x+\frac{T_\beta^ny}{\beta^n},$$
which implies $T_\beta^ny=x'$. Then for any $k\ge 1$,
$$\epsilon_{n+k}(y,\beta)=\lfloor\beta T_\beta^{n+k-1}y\rfloor=\lfloor\beta T_\beta^{k-1}x'\rfloor=\epsilon_k(x',\beta)=w_k'.$$
Thus $\epsilon(y,\beta)=ww'$. Therefore $ww'\in\Sigma_\beta$.
\newline$(3)\Rightarrow(4)$ is obvious.
\newline$(4)\Rightarrow(5)$ follows from $\epsilon_1^*\cdots\epsilon_k^*\in\Sigma_\beta^*$ for any $k\ge 1$.
\newline$(5)\Rightarrow(1)$ We need to prove $T_\beta^nI(w)=[0,1)$. It suffices to show $T_\beta^nI(w)\supset[0,1)$ since the reverse inclusion is obvious. Indeed, let $x\in[0,1)$ and $u=w_1\cdots w_n\epsilon_1(x,\beta)\epsilon_2(x,\beta)\cdots$.
\newline At first, we prove $u\in\Sigma_\beta$. By Lemma \ref{charADM}, it suffices to prove $\sigma^k(u)\prec\epsilon^*(1,\beta)$ for any $k\ge 0$ below.
\newline \textcircled{\small{$1$}} If $k\ge n$, we have
$$\sigma^k(u)=\epsilon_{k-n+1}(x,\beta)\epsilon_{k-n+2}(x,\beta)\cdots=\sigma^{k-n}(\epsilon(x,\beta))\overset{\text{by Lemma \ref{charADM}}}{\prec}\epsilon^*(1,\beta).$$
\newline \textcircled{\small{$2$}} If $0\le k\le n-1$, we have
$$\sigma^k(u)=w_{k+1}\cdots w_n\epsilon_1(x,\beta)\epsilon_2(x,\beta)\cdots.$$
Since $\epsilon(x,\beta)\prec\epsilon^*(1,\beta)$, there exists $m\in\N$ such that $\epsilon_1(x,\beta)=\epsilon_1^*,\cdots,\epsilon_{m-1}(x,\beta)=\epsilon_{m-1}^*$ and $\epsilon_m(x,\beta)<\epsilon_m^*$. Combining $w\epsilon_1^*\cdots\epsilon_m^*\in\Sigma_\beta^*$ and Lemma \ref{charADM}, we get
$$\sigma^k(u)\prec w_{k+1}\cdots w_n\epsilon_1^*\cdots\epsilon_m^*0^\infty=\sigma^k(w\epsilon_1^*\cdots\epsilon_m^*0^\infty)\prec\epsilon^*(1,\beta).$$
Therefore $u\in\Sigma_\beta$.
\newline Let $y\in[0,1)$ such that $\epsilon(y,\beta)=u$. Then $y\in I(w)$. Since
$$\epsilon_k(T_\beta^ny,\beta)=\lfloor\beta T_\beta^{n+k-1}y\rfloor=\epsilon_{n+k}(y,\beta)=\epsilon_k(x,\beta)\quad\text{for any $k\in\N$},$$
we get $x=T_\beta^ny\in T_\beta^nI(w)$.
\newline$(1)\Leftrightarrow(6)$ follows from the facts that the function $\epsilon(\cdot,\beta):[0,1)\rightarrow\Sigma_\beta$ is bijective and the commutativity $\epsilon(T_\beta x,\beta)=\sigma(\epsilon(x,\beta))$.
\end{proof}

\begin{proposition}\label{fullzero}
Let $w, w'\in\Sigma_\beta^*$ be full and $|w|=n\in\N$. Then
\begin{itemize}
\item[\emph{(1)}] the word $ww'$ is full;
\item[\emph{(2)}] the word $\sigma^k (w):=w_{k+1}\cdots w_n$ is full for any $1\le k<n$ ;
\item[\emph{(3)}] the digit $w_n<\lfloor\beta\rfloor$ if $\beta\notin\N$. In particular, $w_n=0$ if $1<\beta<2$.
\end{itemize}
\end{proposition}
\begin{proof}
\begin{itemize}
\item[(1)] A proof has been given in \cite{BuWa14}. We give another proof here to be self-contained. Since $w'$ is full, by Proposition \ref{charFULL} (5) we get $w'\epsilon_1^*\cdots\epsilon_m^*\in\Sigma_\beta^*$  for any $m\ge 1$. Then $ww'\epsilon_1^*\cdots\epsilon_m^*\in\Sigma_\beta^*$ by the fullness of $w$ and Proposition \ref{charFULL} (4), which implies that $ww'$ is full by Proposition \ref{charFULL} (5).
\item[(2)] Since $w$ is full , by Proposition \ref{charFULL} (5) we get $w_1\cdots w_n\epsilon_1^*\cdots\epsilon_m^*\in\Sigma_\beta^*$, also $w_{k+1}\cdots w_n\epsilon_1^*\cdots\epsilon_m^*$ $\in\Sigma_\beta^*$ for any $m\ge 1$. Therefore $w_{k+1}\cdots w_n$ is full by Proposition \ref{charFULL} (5).
\item[(3)] Since $w$ is full, by (2) we know that $\sigma^{n-1}w=w_n$ is full. Then $|I(w_n)|=1/\beta$ by Proposition \ref{charFULL} (2). Suppose $w_n=\lfloor\beta\rfloor$, then $I(w_n)=I(\lfloor\beta\rfloor)=[\lfloor\beta\rfloor/\beta,1)$ and $|I(w_n)|=1-\lfloor\beta\rfloor/\beta<1/\beta$ which is a contradiction. Therefore $w_n\neq\lfloor\beta\rfloor$. So $w_n<\lfloor\beta\rfloor$ noting that $w_n\le \lfloor\beta\rfloor$.
\end{itemize}
\end{proof}

\begin{proposition}\label{exaDIV}
(1) Any truncation of $\epsilon(1,\beta)$ is not full (if it is admissible). That is, $\epsilon(1,\beta)|_k$ is not full for any $k \in \N$ (if it is admissible).
\item(2) Let $k \in \mathbb{N}$. Then $\epsilon^*(1,\beta)|_k$ is full if and only if $\epsilon(1,\beta)$ is finite with length $M$ which exactly divides $k$, i.e., $M|k$.
\end{proposition}
\begin{proof}(1) We show the conclusion by the cases that $\epsilon(1,\beta)$ is finite or infinite.
\newline Cases 1. $\epsilon(1,\beta)$ is finite with length $M$.
\newline \textcircled{\small{$1$}} If $k\ge M$, then $\epsilon(1,\beta)|_k=\epsilon_1\cdots\epsilon_M 0^{k-M}$ is not admissible.
\newline \textcircled{\small{$2$}} If $1\le k \le M-1$, combining $\epsilon_{k+1}\cdots\epsilon_M0^\infty=\epsilon(T_\beta^k1,\beta)\in\Sigma_\beta$, $\epsilon_1\cdots\epsilon_k\epsilon_{k+1}\cdots\epsilon_M0^\infty=\epsilon(1,\beta)\notin\Sigma_\beta$ and Proposition \ref{charFULL} (1) (3), we know that $\epsilon(1,\beta)|_k=\epsilon_1\cdots\epsilon_k$ is not full.
\newline Cases 2. $\epsilon(1,\beta)$ is infinite. It follows from the similar proof with Case 1 \textcircled{\small{$2$}}.
\item(2) $\boxed{\Leftarrow}$ Let $p\in\mathbb{N}$ with $k=pM$. For any $n\ge 1$, we know that $\epsilon_1^* \cdots \epsilon_{pM}^* \epsilon_1^* \cdots \epsilon_n^*=\epsilon^*(1,\beta)|_{k+n}$ is admissible by Lemma \ref{1cutADM}. Therefore $\epsilon^*(1,\beta)|_k=\epsilon_1^* \cdots \epsilon_{pM}^*$ is full by Proposition \ref{charFULL} (1) (5).
\newline$\boxed{\Rightarrow}$ (By contradiction) Suppose that the conclusion is not true, that is, either $\epsilon(1,\beta)$ is infinite or finite with length $M$, but $M$ does not divide $k$ exactly.
\newline \textcircled{\small{$1$}} If $\epsilon(1,\beta)$ is infinite, then $\epsilon^*(1,\beta)|_k=\epsilon(1,\beta)|_k$ is not full by (1), which contradicts our condition.
\newline \textcircled{\small{$2$}} If $\epsilon(1,\beta)$ is finite with length $M$, but $M\nmid k$, then there exists $p\ge 0$ such that $pM<k<pM+M$. Since $\epsilon^*(1,\beta)|_k$ is full, combining
$$\epsilon_{k-pM+1}\cdots\epsilon_M0^\infty=\epsilon(T_\beta^{k-pM}1,\beta)\in\Sigma_\beta,$$
and Proposition \ref{charFULL} (1) (3), we get $\epsilon_1^*\cdots\epsilon_k^*\epsilon_{k-pM+1}\cdots\epsilon_{M-1}\epsilon_M0^\infty\in\Sigma_\beta$,
i.e., $\epsilon_1^*\cdots\epsilon_{pM}^*\epsilon_1\cdots\epsilon_{M-1}\epsilon_M0^\infty$ $\in\Sigma_\beta$ which is false since $\pi_\beta(\epsilon_1^*\cdots\epsilon_{pM}^*\epsilon_1\cdots\epsilon_{M-1}\epsilon_M0^\infty)=1$.
\end{proof}

The following lemma is a convenient way to show that an admissible word is not full.

\begin{lemma}\label{tailnotfull}Any admissible word ends with a prefix of $\epsilon(1,\beta)$ is not full. That is, if there exists $1\le s\le n$ such that $w=w_1\cdots w_{n-s}\epsilon_1\cdots\epsilon_s\in\Sigma_\beta^n$, then $w$ is not full.
\end{lemma}
\begin{proof} It follows from Proposition \ref{fullzero} (2) and Proposition \ref{exaDIV} (1).
\end{proof}

\begin{notation} Denote the first position where $w$ and $\epsilon(1,\beta)$ are different by
$$\mathfrak{m}(w) := \min\{k \ge 1 : w_k < \epsilon_k \} \quad\text{for } w\in \Sigma_\beta $$
and
$$\mathfrak{m}(w) := \mathfrak{m}(w0^\infty) \quad\text{for } w \in \Sigma^*_\beta.$$
\end{notation}

\begin{remark}\label{compare}
(1) Let $\epsilon(1,\beta)$ be finite with the length $M$. Then $\mathfrak{m}(w)\le M$ for any $w$ in $\Sigma_\beta$ or $\Sigma_\beta^*$.
\item(2) Let $w\in\Sigma_\beta^n$ and $\mathfrak{m}(w)\ge n$. Then $w=\epsilon_1 \cdots \epsilon_{n-1} w_n$ with $w_n \le \epsilon_n$.
\end{remark}
\begin{proof}(1) follows from $w\prec\epsilon(1,\beta)$.
\item(2) follows from $w_1=\epsilon_1, \cdots , w_{n-1}=\epsilon_{n-1}$ and $w \in \Sigma_\beta^n$.
\end{proof}

We give the complete characterizations of the structures of admissible words, full words and non-full words by the following two theorems and a corollary as basic results of this paper.

\begin{theorem}[The structure of admissible words]\label{strADM}
Let $w \in \Sigma_\beta^n$. Then $w=w_1w_2\cdots w_n$ can be uniquely decomposed to the form
\begin{align}\label{sumN1}\epsilon_1\cdots \epsilon_{k_1-1} w_{n_1}\epsilon_1\cdots \epsilon_{k_2-1} w_{n_2}\cdots\epsilon_1\cdots\epsilon_{k_p-1} w_{n_p} \epsilon_1 \cdots \epsilon_{l-1} w_n,\end{align}
where $p\ge0$, $k_1,\cdots,k_p,l\in \mathbb{N}$, $n=k_1+...+k_p+l$, $n_j=k_1+\cdots+k_j$, $w_{n_j}<\epsilon_{k_j}$ for all $1 \le j \le p$, $w_n \le \epsilon _l$ and the words $\epsilon_1\cdots \epsilon_{k_1-1} w_{n_1},\cdots,\epsilon_1\cdots\epsilon_{k_p-1} w_{n_p}$ are all full.
\newline\indent Moreover, if $\epsilon(1,\beta)$ is finite with length $M$, then $k_1,\cdots,k_p,l\le M$. For the case $l=M$, we must have $w_n< \epsilon_M$.
\end{theorem}

\begin{theorem}[The structural criterion of full words]\label{strFULL}
Let $w \in \Sigma_\beta^n$ and $w_*:=\epsilon_1 \cdots \epsilon_{l-1} w_n$ be the suffix of $w$ as in Theorem \ref{strADM}. Then
$$w \text{ is full} \Longleftrightarrow w_* \text{ is full} \Longleftrightarrow w_n<\epsilon_{|w_*|}.$$
\end{theorem}

\begin{corollary}\label{tail-non-full}Let $w\in\Sigma_\beta^n$. Then $w$ is not full if and only if it ends with a prefix of $\epsilon(1,\beta)$. That is, when $\epsilon(1,\beta)$ is infinite (finite with length $M$), there exists $1\le s\le n$ ( $1\le s\le \min\{M-1,n\}$ respectively) such that $w=w_1\cdots w_{n-s}\epsilon_1\cdots\epsilon_s$.
\end{corollary}

\begin{proof}
$\boxed{\Rightarrow}$ follows from Theorem \ref{strADM} and Theorem \ref{strFULL}.
$\boxed{\Leftarrow}$ follows from Lemma \ref{tailnotfull}.
\end{proof}

\begin{proof}[Proof of Theorem \ref{strADM}]Firstly, we show the decomposition by the cases that $\epsilon(1,\beta)$ is infinite or finite.
\newline Case 1. $\epsilon(1,\beta)$ is infinite.
\newline Compare $w$ and $\epsilon(1,\beta)$. If $\mathfrak{m}(w)\ge n$, then $w$ has the form (\ref{sumN1}) with $w=\epsilon_1 \cdots \epsilon_{n-1} w_n$ by Remark \ref{compare} (2). If $\mathfrak{m}(w) < n$, let $n_1=k_1=\mathfrak{m}(w)\ge 1$. Then $w|_{n_1}=\epsilon_1\cdots\epsilon_{k_1-1} w_{n_1}$ with $w_{n_1}<\epsilon_{k_1}$.
Continue to compare the tail of $w$ and $\epsilon(1,\beta)$. If $\mathfrak{m}(w_{n_1+1}\cdots w_n)\ge n-n_1$, then $w_{n_1+1}\cdots w_n=\epsilon_1 \cdots \epsilon_{n-n_1-1} w_n$ with $w_n \le \epsilon_{n-n_1}$ by Remark \ref{compare} (2) and $w$ has the form (\ref{sumN1}) with $w=\epsilon_1 \cdots\epsilon_{k_1-1}w_{n_1}\epsilon_1\cdots\epsilon_{n-n_1-1}w_n$. If $\mathfrak{m}(w_{n_1+1}\cdots w_n)< n-n_1$, let $k_2=\mathfrak{m}(w_{n_1+1}\cdots w_n)\ge 1$ and $n_2=n_1+k_2$. Then $w|_{n_2}=\epsilon_1\cdots\epsilon_{k_1-1} w_{n_1} \epsilon_1\cdots\epsilon_{k_2-1} w_{n_2}$ with $w_{n_2}<\epsilon_{k_2}$.
Continue to compare the tail of $w$ and $\epsilon(1,\beta)$ for finite times. Then we can get that $w$ must have the form (\ref{sumN1}).
\newline Case 2. $\epsilon(1,\beta)$ is finite with length $M$.
\newline By Remark \ref{compare}(1), we get $\mathfrak{m}(w)$,$\mathfrak{m}(w_{n_1+1}\cdots w_n)$, $\cdots$, $\mathfrak{m}(w_{n_j+1}\cdots w_n)$, $\cdots$, $\mathfrak{m}(w_{n_p+1}\cdots w_n) \le M $ in Case 1. That is, $k_1,k_2,\cdots,k_p,l \le M$ in (\ref{sumN1}). For the case $l=M$, combining $w_{n_p+1}=\epsilon_1, \cdots, w_{n-1}=\epsilon_{M-1}$ and $w_{n_p+1}\cdots w_n \prec \epsilon_1\cdots \epsilon_M$, we get $w_n<\epsilon_M$.
\newline Secondly, $\epsilon_1\cdots \epsilon_{k_1-1} w_{n_1},\cdots,\epsilon_1\cdots\epsilon_{k_p-1} w_{n_p}$ are obviously full by Lemma \ref{admfull}.
\end{proof}

\begin{proof}[Proof of Theorem \ref{strFULL}]
By Proposition \ref{fullzero} (1) (2), we know that $w$ is full $\Longleftrightarrow w_*$ is full. So it suffices to prove that $w_*$ is full $\Longleftrightarrow w_n<\epsilon_{|w_*|}$.
\newline$\boxed{\Rightarrow}$ By $w_* \in \Sigma_\beta^*$, we get $w_n \le \epsilon_l$. Suppose $w_n=\epsilon_l$, then $w_*=\epsilon_1 \cdots \epsilon_l$ is not full by Proposition \ref{exaDIV} (1), which contradicts our condition. Therefore $w_n<\epsilon_l$.
\newline$\boxed{\Leftarrow}$ Let $w_n<\epsilon_l$. We show that $w_*$ is full by the cases that $\epsilon(1,\beta)$ is infinite or finite.
\newline Case 1. When $\epsilon(1,\beta)$ is infinite. we know that $w_*$ is full by $\epsilon_1 \cdots \epsilon_{l-1} \epsilon_l \in \Sigma_\beta^*, w_n<\epsilon_l$ and Lemma \ref{admfull}.
\newline Case 2. When $\epsilon(1,\beta)$ is finite with length $M$, we know $l\le M$ by Theorem \ref{strADM}.
\newline If $l<M$, we get $\epsilon_1 \cdots \epsilon_{l-1} \epsilon_l\in \Sigma_\beta^*$. Then $w_*$ is full by $w_n<\epsilon_l$ and Lemma \ref{admfull}.
\newline If $l=M$, we know that $\epsilon_1 \cdots \epsilon_{l-1} (\epsilon_l-1) = \epsilon_1 \cdots \epsilon_{M-1} (\epsilon_M-1) = \epsilon_1^* \cdots \epsilon_M^*$ is full by Proposition \ref{exaDIV} (2). Then $w_*$ is full by $w_n \le \epsilon_l-1$ and Lemma \ref{admfull}.
\end{proof}

From Theorem \ref{strADM}, Theorem \ref{strFULL} and Corollary \ref{tail-non-full} above, we can understand the structures of admissible words, full words and non-full words clearly, and judge whether an admissible word is full or not conveniently. They will be used for many times in the following sections.

\section{The lengths of the runs of full words}

\begin{definition}\label{nonzerosequence}Let $\beta>1$. Define $\{n_i(\beta)\}$ to be those positions of $\epsilon(1,\beta)$ that are nonzero. That is,
$$n_1(\beta):=\min\{k \ge 1: \epsilon_k \neq 0\} \text{ and } n_{i+1}(\beta):=\min\{k>n_i:\epsilon_k \neq 0\}$$
if there exists $k>n_i$ such that $\epsilon_k\neq0$ for $i\ge1$. We call $\{n_i(\beta)\}$ the \textit{nonzero sequence} of $\beta$, also denote it by $\{n_i\}$ if there is no confusion.
\end{definition}

\begin{remark}Let $\beta>1$, $\{n_i\}$ be the nonzero sequence of $\beta$. Then the followings are obviously true.
\begin{itemize}
\item[(1)] $n_1=1$;
\item[(2)] $\epsilon(1,\beta)$ is finite if and only if $\{n_i\}$ is finite;
\item[(3)] $\epsilon(1,\beta)=\epsilon_{n_1} 0 \cdots 0 \epsilon_{n_2} 0 \cdots 0 \epsilon_{n_3} 0 \cdots$.
\end{itemize}
\end{remark}

\begin{definition}\label{maximal}(1) Denote by $[w^{(1)},\cdots,w^{(l)}]$ the $l$ consecutive words from small to large in $\Sigma_\beta^n$ with lexicographic order, which is called \textit{a run of words} and $l$  is the \textit{length} of the run of words. If $w^{(1)},\cdots,w^{(l)}$ are all full, we call $[w^{(1)},\cdots,w^{(l)}]$ \textit{a run of full words}.
\newline(2) A run of full words $[w^{(1)},\cdots,w^{(l)}]$ is said to be \textit{maximal}, if it can not be elongated, i.e., `` the previous word of $w^{(1)}$ in $\Sigma_\beta^n$ is not full or $w^{(1)}=0^n$ '' and `` the next word of $w^{(l)}$ is not full or $w^{(l)}=\epsilon^*(1,\beta)|_n$ ''.
\newline In a similar way, we can define \textit{a run of non-full words} and \textit{a maximal run of non-full words}.
\end{definition}

\begin{definition}\label{set}
We use $\mathcal{F}_\beta^n$ to denote the set of all the maximal runs of full words in $\Sigma_\beta^n$ and $F_\beta^n$ to denote the length set of $\mathcal{F}_\beta^n$, i.e.,
$$F_\beta^n:=\{l\in\N: \text{ there exists } [w^{(1)},\cdots,w^{(l)}]\in\mathcal{F}_\beta^n \}.$$
Similarly, we use $\mathcal{N}_\beta^n$ to denote the set of all the maximal runs of non-full words and $N_\beta^n$ to denote the length set of $\mathcal{N}_\beta^n$.
\newline In $\mathcal{F}_\beta^n\cup\mathcal{N}_\beta^n$, we use $S_{max}^n$ to denote the maximal run with $\epsilon^*(1,\beta)|_n$ as its last element.
\end{definition}

\begin{remark}For any $w\in\Sigma_\beta^n$ with $w\neq 0^n$ and $w_n=0$, the previous word of $w$ in the lexicographic order in $\Sigma_\beta^n$ is $w_1\cdots w_{k-1}(w_k-1)\epsilon_1^*\cdots\epsilon_{n-k}^*$ where $k=\max\{1\le i\le n-1:w_i\neq0\}$.
\end{remark}

Notice that we will use the basic fact above for many times in the proofs of the following results in this paper.

\begin{theorem}[The lengths of the maximal runs of full words]\label{fulllength}Let $\beta>1$ with $\beta\notin\N$, $\{n_i\}$ be the nonzero sequence of $\beta$. Then
$$F_\beta^n=\left\{\begin{array}{ll}
\{\epsilon_{n_i}:n_i\le n\} & \mbox{if } \epsilon(1,\beta) \mbox{ is infinite or finite with length } M\ge n; \\
\{\epsilon_{n_i}\}\cup\{\epsilon_1+\epsilon_M\} & \mbox{if } \epsilon(1,\beta) \mbox{ is finite with length } M<n \mbox{ amd } M|n; \\
\{\epsilon_{n_i}:n_i\neq M\}\cup\{\epsilon_1+\epsilon_M\} & \mbox{if } \epsilon(1,\beta) \mbox{ is finite with length } M<n \mbox{ and } M\nmid n.
\end{array}\right.$$
\end{theorem}
\begin{proof}It follows from Definition \ref{maximal}, Lemma \ref{fulllength1}, Lemma \ref{fulllength2} and the fact that $n_i\le M$ for any $i$ when $\epsilon(1,\beta)$ is finite with length $M$.
\end{proof}

\begin{remark}By Theorem \ref{fulllength}, when $1<\beta<2$, we have
$$F_\beta^n=\left\{\begin{array}{ll}
\{1\} & \mbox{if } \epsilon(1,\beta) \mbox{ is infinite or finite with length } M\ge n; \\
\{1,2\} & \mbox{if } \epsilon(1,\beta) \mbox{ is finite with length } M<n.
\end{array}\right.$$
\end{remark}

\begin{lemma}\label{fulllength1}Let $\beta>1$ with $\beta\notin\N$, $\{n_i\}$ be the nonzero sequence of $\beta$. Then the length set of $\mathcal{F}_\beta^n \backslash \{S_{max}^n\}$, i.e., $\{l\in\N: \text{ there exists } [w^{(1)},\cdots,w^{(l)}]\in\mathcal{F}_\beta^n \backslash \{S_{max}^n\}\}$ is
$$\left\{\begin{array}{ll}
\{\epsilon_{n_i}:n_i\le n\} & \mbox{if } \epsilon(1,\beta) \mbox{ is infinite or finite with length } M>n; \\
\{\epsilon_{n_i}:n_i\neq M\} & \mbox{if } \epsilon(1,\beta) \mbox{ is finite with length } M=n; \\
\{\epsilon_{n_i}:n_i\neq M\}\cup\{\epsilon_1+\epsilon_M\} & \mbox{if } \epsilon(1,\beta) \mbox{ is finite with length } M<n.
\end{array}\right.$$
\end{lemma}
\begin{proof}Let $[w^{(l)},w^{(l-1)},\cdots,w^{(2)},w^{(1)}]\in\mathcal{F}_\beta^n \backslash \{S_{max}^n\}$ and $w$ which is not full be the next word of $w^{(1)}$. By Corollary \ref{tail-non-full}, there exist $1 \le s \le n$, $0 \le a \le n-1$ with $a+s=n$ ($s \le M-1$, when $\epsilon(1,\beta)$ is finite with length $M$), such that $w=w_1\cdots w_a\epsilon_1\cdots\epsilon_s$.
\item(1) If $s=1$, that is, $w=w_1\cdots w_{n-1}\epsilon_1$, then
$w^{(1)}=w_1\cdots w_{n-1}(\epsilon_1-1)$,
$w^{(2)}=w_1\cdots w_{n-1}(\epsilon_1-2)$,
$\cdots$,
$w^{(\epsilon_1)}=w_1\cdots w_{n-1}0$ are full by Lemma \ref{admfull}.
\newline \textcircled{\small{$1$}} If $n=1$ or $w_1\cdots w_{n-1}=0^{n-1}$, it is obvious that $l=\epsilon_1$.
\newline \textcircled{\small{$2$}} If $n\ge 2$ and $w_1\cdots w_{n-1}\neq 0^{n-1}$, there exists $1 \le k \le n-1$ such that $w_k\neq 0$ and $w_{k+1}=\cdots=w_{n-1}=0$. Then the previous word of $w^{(\epsilon_1)}$ is
$$w^{(\epsilon_1+1)}=w_1\cdots w_{k-1}(w_k-1)\epsilon_1^*\cdots\epsilon_{n-k}^*.$$
\begin{itemize}
\item[i)] If $\epsilon(1,\beta)$ is infinite or finite with length $M\ge n$, then $w^{(\epsilon_1+1)}=w_1\cdots w_{k-1}(w_k-1)\epsilon_1\cdots\epsilon_{n-k}$ is not full by Lemma \ref{tailnotfull}. Therefore $l=\epsilon_1$.
\item[ii)] If $\epsilon(1,\beta)$ is finite with length $M<n$, we divide this case into two parts according to $M\nmid n-k$ or $M|n-k$.
\newline\textcircled{\small{$a$}} If $M\nmid n-k$, then $\epsilon_1^*\cdots \epsilon_{n-k}^*$ is not full by Proposition \ref{exaDIV} (2) and $w^{(\epsilon_1+1)}$ is also not full by Proposition \ref{fullzero} (2). Therefore $l=\epsilon_1$.
\newline\textcircled{\small{$b$}} If $M|n-k$, then $\epsilon_1^*\cdots \epsilon_{n-k}^*$ is full by Proposition \ref{exaDIV} (2) and $w^{(\epsilon_1+1)}$ is also full by Lemma \ref{admfull} and Proposition \ref{fullzero} (1). Let $w_1'\cdots w_{n-M}':=w_1\cdots w_{k-1}(w_k-1)\epsilon_1^*\cdots\epsilon_{n-k-M}^*$. Then
$$\begin{array}{rl}
w^{(\epsilon_1+1)}&=w_1'\cdots w_{n-M}'\epsilon_1\cdots\epsilon_{M-1}(\epsilon_M-1).
\end{array}$$
The consecutive previous words
$$\begin{array}{rl}
w^{(\epsilon_1+2)}&=w_1'\cdots w_{n-M}'\epsilon_1\cdots\epsilon_{M-1}(\epsilon_M-2)\\
w^{(\epsilon_1+3)}&=w_1'\cdots w_{n-M}'\epsilon_1\cdots\epsilon_{M-1}(\epsilon_M-3)\\
&\cdots\\
w^{(\epsilon_1+\epsilon_M)}&=w_1'\cdots w_{n-M}'\epsilon_1\cdots\epsilon_{M-1}0
\end{array}$$
are all full by Lemma \ref{admfull}. Since $\epsilon_1\neq0$ and $M>1$, there exists $1\le t\le M-1$ such that $\epsilon_t\neq0$ and $\epsilon_{t+1}=\cdots=\epsilon_{M-1}=0$. Then, as the previous word of $w^{(\epsilon_1+\epsilon_M)}$,
$$w^{(\epsilon_1+\epsilon_M+1)}=w_1'\cdots w_{n-M}'\epsilon_1\cdots\epsilon_{t-1}(\epsilon_t-1)\epsilon_1\cdots\epsilon_{M-t}$$
is not full by Lemma \ref{tailnotfull}. Therefore $l=\epsilon_1+\epsilon_M$.
\end{itemize}
\item(2) If $2\le s\le n$, we divide this case into two parts according to $\epsilon_s=0$ or not.
\newline \textcircled{\small{$1$}} If $\epsilon_s=0$, there exists $1\le t\le s-1$ such that $\epsilon_t\neq0$ and $\epsilon_{t+1}=\cdots=\epsilon_s=0$ by $\epsilon_1\neq0$. Then $w=w_1\cdots w_a\epsilon_1\cdots\epsilon_t0^{s-t}$, and $w^{(1)}=w_1\cdots w_a\epsilon_1\cdots\epsilon_{t-1}(\epsilon_t-1)\epsilon_1\cdots\epsilon_{s-t}$ is not full by Lemma \ref{tailnotfull}, which contradicts our assumption.
\newline \textcircled{\small{$2$}} If $\epsilon_s\neq0$, then
$$\begin{array}{rl}
w^{(1)}&=w_1\cdots w_a\epsilon_1\cdots\epsilon_{s-1}(\epsilon_s-1)\\
w^{(2)}&=w_1\cdots w_a\epsilon_1\cdots\epsilon_{s-1}(\epsilon_s-2)\\
&\cdots\\
w^{(\epsilon_s)}&=w_1\cdots w_a\epsilon_1\cdots\epsilon_{s-1}0
\end{array}$$
are full by Lemma \ref{admfull}.
By nearly the same way of \textcircled{\small{$1$}}, we can prove that the previous word of $w^{(\epsilon_s)}$ is not full. Therefore $l=\epsilon_s$.
\begin{itemize}
\item[i)] If $\epsilon(1,\beta)$ is infinite or finite with length $M>n$, combining $2\le s\le n$ and $\epsilon_s\neq0$, we know that the set of all values of $l=\epsilon_s$ is $\{\epsilon_{n_i}:2\le n_i\le n\}$.
\item[ii)] If $\epsilon(1,\beta)$ finite with length $M\le n$, combining $2\le s\le M-1$ and $\epsilon_s\neq0$, we know that the set of all values of $l=\epsilon_s$ is $\{\epsilon_{n_i}:2\le n_i<M\}$.
\end{itemize}
By the discussion above, we can see that in every case, every value of $l$ can be achieved. Combining $n_i\le M$ for any $i$ when $\epsilon(1,\beta)$ is finite with length $M$, $\epsilon_{n_1}=\epsilon_1$ and all the cases discussed above, we get the conclusion of this lemma.
\end{proof}

\begin{lemma}\label{fulllength2}Let $\beta>1$ with $\beta\notin\N$. If $\epsilon(1,\beta)$ is finite with length $M$ and $M|n$, then $S_{max}^n\in\mathcal{F}_\beta^n$ and the length of $S_{max}^n$ is $\epsilon_M$. Otherwise, $S_{max}^n\in\mathcal{N}_\beta^n$.
\end{lemma}
\begin{proof}Let $w^{(1)}=\epsilon_1^*\cdots\epsilon_n^*$.
\newline If $\epsilon(1,\beta)$ is finite with length $M$ and $M|n$, then $w^{(1)}$ is full by Proposition \ref{exaDIV} (2). We get $S_{max}^n\in\mathcal{F}_\beta^n$. Let $p=n/m-1\ge0$. As the consecutive previous words of $w^{(1)}$, $w^{(2)}=(\epsilon_1\cdots\epsilon_{M-1}(\epsilon_M-1))^p\epsilon_1\cdots\epsilon_{M-1}(\epsilon_M-2)$,
$\cdots$,
$w^{(\epsilon_M)}=(\epsilon_1\cdots\epsilon_{M-1}(\epsilon_M-1))^p\epsilon_1\cdots\epsilon_{M-1}0$ are full by Lemma \ref{admfull}.
By nearly the same way in the proof of Lemma \ref{fulllength1} (2) \textcircled{\small{$1$}}, we know that the previous word of $w^{(\epsilon_M)}$ is not full. Therefore the number of $S_{max}^n$ is $\epsilon_M$.
\newline Otherwise, $w^{(1)}$ is not full by Proposition \ref{exaDIV} (2). We get $S_{max}^n\in\mathcal{N}_\beta^n$.
\end{proof}

\begin{remark}All the locations of all the lengths in Theorem \ref{fulllength} can be found in the proof of Lemma \ref{fulllength1} and Lemma \ref{fulllength2}.
\end{remark}


\begin{corollary}[The maximal length of the runs of full words]\label{maxlengthfull}Let $\beta>1$ with $\beta\notin\N$. Then
$$\max F_\beta^n=\left\{\begin{array}{ll}
\lfloor\beta\rfloor+\epsilon_M & \mbox{if } \epsilon(1,\beta) \mbox{ is finite with length } M<n;\\
\lfloor\beta\rfloor & \mbox{if } \epsilon(1,\beta) \mbox{ is infinite or finite with length } M\ge n.
\end{array}\right.$$
\end{corollary}
\begin{proof}It follows from $\epsilon_{n_i}\le\epsilon_{n_1}=\epsilon_1=\lfloor\beta\rfloor$ for any $i$ and Theorem \ref{fulllength}.
\end{proof}

\begin{corollary}[The minimal length of the maximal runs of full words]\label{minlengthfull}Let $\beta>1$ with $\beta\notin\N$, $\{n_i\}$ be the nonzero sequence of $\beta$. Then
$$\min F_\beta^n=\left\{\begin{array}{ll}
\min\limits_{n_i<M}\epsilon_{n_i }& \mbox{if } \epsilon(1,\beta) \mbox{ is finite with length } M<n \mbox{ and } M\nmid n;\\
\min\limits_{n_i\le n}\epsilon_{n_i} & \mbox{otherwise.}
\end{array}\right.$$
\end{corollary}
\begin{proof}It follows from $n_i\le M$ for any $i$ when $\epsilon(1,\beta)$ is finite with length $M$ and Theorem \ref{fulllength}.
\end{proof}

\begin{remark}It follows from Theorem \ref{fulllength} that the lengths of maximal runs of full words rely on the nonzero terms in $\epsilon(1,\beta)$, i.e., $\{\epsilon_{n_i}\}$.
\end{remark}

\section{The lengths of runs of non-full words}

Let $\{n_i\}$ be the nonzero sequence of $\beta$. We will use a similar concept of numeration system and greedy algorithm in the sense of \cite[Section 3.1]{AlSh03} to define the function $\tau_\beta$ below. For any $s\in\N$, we can write $s=\sum_{i\ge1}a_i n_i$ greedily and uniquely where $a_i\in\N\cup\{0\}$ for any $i$ and then define $\tau_\beta(s)=\sum_{i\ge_1}a_i$. Equivalently, we have the following.

\begin{definition}[The function $\tau_\beta$]\label{tau}Let $\beta>1$, $\{n_i\}$ be the nonzero sequence of $\beta$ and $s \in \N$. Define $\tau_\beta(s)$ to be the number needed to add up to $s$ greedily  by $\{n_i\}$ with repetition. We define it precisely below.
\newline Let $n_{i_1}=\max\{n_i:n_i\le s\}$. (Notice $n_1=1$.)
\newline If $n_{i_1}=s$, define $\tau_\beta(s):=1$.
\newline If $n_{i_1}<s$, let $t_1=s-n_{i_1}$ and $n_{i_2}=\max\{n_i:n_i\le t_1\}$.
\newline \indent\indent\indent\indent\indent If $n_{i_2}=t_1$, define $\tau_\beta(s):=2$.
\newline \indent\indent\indent\indent\indent If $n_{i_2}<t_1$, let $t_2=t_1-n_{i_2}$ and $n_{i_3}=\max\{n_i:n_i\le t_2\}$.
\newline $\cdots$
\newline$\begin{array}{ll}
\mbox{Generally for } j\in\N. & \mbox{ If } n_{i_j}=t_{j-1}(t_0:=s), \mbox{ define } \tau_\beta(s):=j.\\
                              & \mbox{ If } n_{i_j}<t_{j-1}, \mbox{ let } t_j=t_{j-1}-n_{i_j} \mbox{ and } n_{i_{j+1}}=\max\{n_i:n_i\le t_j\}.
\end{array}$
\newline $\cdots$
\newline Noting that $n_1=1$, it is obvious that there exist $n_{i_1}\ge n_{i_2}\ge \cdots \ge n_{i_d}$ all in $\{n_i\}$ such that $s=n_{i_1}+n_{i_2}+\cdots+n_{i_d}$, i.e., $n_{i_d}=t_{d-1}$. Define $\tau_\beta(s):=d$.
\end{definition}

In the following we give an example to show how to calculate $\tau_\beta$.

\begin{example} Let $\beta>1$ such that $\epsilon(1,\beta)=302000010^\infty$ (such $\beta$ exists by Lemma \ref{1exp}). Then the nonzero sequence of $\beta$ is $\{1,3,8\}$. The way to add up to $7$ greedily with repetition is $7=3+3+1$. Therefore $\tau_\beta(7)=3$.
\end{example}

\begin{proposition}[Properties of $\tau_\beta$]\label{pro}Let $\beta>1$, $\{n_i\}$ be the nonzero sequence of $\beta$ and $n\in\N$. Then
\begin{itemize}
\item[\emph{(1)}] $\tau_\beta(n_i)=1$ for any $i$;
\item[\emph{(2)}] $\tau_\beta(s)=s$ for any $1\le s\le n_2-1$, and $\tau_\beta(s)\le s$ for any $s\in\N$;
\item[\emph{(3)}] $\{1,2,\cdots,k\}\subset\{\tau_\beta(s):1\le s\le n\}$ for any $k\in \{\tau_\beta(s):1\le s\le n\}$;
\item[\emph{(4)}] $\{\tau_\beta(s):1\le s\le n\}=\{1,2,\cdots,\max\limits_{1\le s\le n}\tau_\beta(s)\}$.
\end{itemize}
\end{proposition}
\begin{proof}(1) and (2) follow from Definition \ref{tau} and $n_1=1$.
\begin{itemize}
\item[(3)] Let $k\in \{\tau_\beta(s):1\le s\le n\}$. If $k=1$, the conclusion is obviously true. If $k\ge 2$, let $2\le t_0\le n$ such that $k=\tau_\beta(t_0)$, $n_{i_1}=\max\{n_i:n_i\le t_0\}$ and $t_1=t_0-n_{i_1}$. Then $1\le t_1<t_0\le n$ and it is obvious that $k-1=\tau_\beta(t_1)\in\{\tau_\beta(s):1\le s\le n\}$ by Definition \ref{tau}. By the same way, we can get $k-2,k-3,\cdots,1\in\{\tau_\beta(s):1\le s\le n\}$. Therefore $\{1,2,\cdots,k\}\subset\{\tau_\beta(s):1\le s\le n\}$.
\item[(4)] The inclusion $\{\tau_\beta(s):1\le s\le n\}\subset\{1,2,\cdots,\max\limits_{1\le s\le n}\tau_\beta(s)\}$ is obvious and the reverse inclusion follows from $\max\limits_{1\le s \le n}\tau_\beta(s)\in\{\tau_\beta(s):1\le s\le n\}$ and (3).
\end{itemize}
\end{proof}

For $n\in\N$, we use $r_n(\beta)$ to denote the maximal length of the strings of 0's in $\epsilon^*_1\cdots\epsilon^*_n$ as in \cite{FWL16}, \cite{HTY16} and \cite{TYZ16}, i.e.,
$$r_n(\beta)=\max\{k\ge1:\epsilon^*_{i+1}=\cdots=\epsilon^*_{i+k}=0 \text{ for some } 0\le i\le n-k\}$$
with the convention that $\max\emptyset=0$.

The following relation between $\tau_\beta(s)$ and $r_s(\beta)$ will be used in the proof of Corollary \ref{notfulllengthupperbound}.

\begin{proposition}\label{runlength}Let $\beta>1$. If $\epsilon(1,\beta)$ is infinite, then $\tau_\beta(s)\le r_s(\beta)+1$ for any $s\ge1$. If $\epsilon(1,\beta)$ is finite with length $M$, then $\tau_\beta(s)\le r_s(\beta)+1$ is true for any $1\le s\le M$.
\end{proposition}
\begin{proof}Let $\{n_i\}$ be the nonzero sequence of $\beta$ and $n_{i_1}=\max\{n_i:n_i\le s\}$. No matter $\epsilon(1,\beta)$ is infinite with $s\ge1$ or finite with length $M\ge s\ge1$, we have
$$\tau_\beta(s)-1=\tau_\beta(s-n_{i_1})\le s-n_{i_1}\le r_s(\beta)$$
since $s-n_{i_1}=0$ or $\epsilon^*_{n_{i_1}+1}\epsilon^*_{n_{i_1}+2}\cdots\epsilon^*_s=\epsilon_{n_{i_1}+1}\epsilon_{n_{i_1}+2}\cdots\epsilon_s=0^{s-n_{i_1}}$.
\end{proof}

\begin{lemma}\label{tau-non-full}Let $n\in\N$, $\beta>1$ with $\beta\notin\N$ and $w\in\Sigma_\beta^n$ end with a prefix of $\epsilon(1,\beta)$, i.e., $w=w_1\cdots w_{n-s}\epsilon_1\cdots\epsilon_s$ where $1\le s\le n$. Then the previous consecutive $\tau_\beta(s)$ words starting from $w$ in $\Sigma_\beta^n$ are not full, but the previous $(\tau_\beta(s)+1)$-th word is full.
\end{lemma}

\begin{remark}Notice that $w=w_1\cdots w_{n-s}\epsilon_1\cdots\epsilon_s$ does not imply that $w_1\cdots w_{n-s}$ is full. For example, when $\beta>1$ with $\epsilon(1,\beta)=1010010^\infty$, let $w=001010=w_1\cdots w_4\epsilon_1\epsilon_2$. But $w_1\cdots w_4=0010$ is not full by Lemma \ref{tailnotfull}.
\end{remark}

\begin{proof}[Proof of Lemma \ref{tau-non-full}]Let $\{n_i\}$ be the nonzero sequence of $\beta$ and
$$w^{(1)}:=w_1^{(1)}\cdots w_{a_1}^{(1)}\epsilon_1\cdots\epsilon_s:=w_1\cdots w_{n-s}\epsilon_1\cdots\epsilon_s=w,$$
where $a_1=n-s$. It is not full by Lemma \ref{tailnotfull}.
\newline$\cdots$
\newline Generally for any $j\ge1$, suppose $w^{(j)},w^{(j-1)},\cdots,w^{(2)},w^{(1)}$ to be $j$ consecutive non-full words in $\Sigma_\beta^n$ where $w^{(j)}=w_1^{(j)}\cdots w_{a_j}^{(j)}\epsilon_1\cdots\epsilon_{t_{j-1}}$, $t_{j-1}>0$ $(t_0:=s)$. Let $w^{(j+1)}\in \Sigma_\beta^n$ be the previous word of $w^{(j)}$ and $n_{i_j}:=\max\{n_i:n_i\le t_{j-1}\}$.
\newline If $n_{i_j}=t_{j-1}$, then $\epsilon_{t_{j-1}}>0$ and $w^{(j+1)}=w_1^{(j)}\cdots w_{a_j}^{(j)}\epsilon_1\cdots\epsilon_{t_{j-1}-1}(\epsilon_{t_{j-1}}-1)$ is full by Lemma \ref{admfull}. We get the conclusion of this lemma since $\tau_\beta(s)=j$ at this time.
\newline If $n_{i_j}<t_{j-1}$, let $t_j=t_{j-1}-n_{i_j}$. Then $w^{(j)}=w_1^{(j)}\cdots w_{a_j}^{(j)}\epsilon_1\cdots\epsilon_{n_{i_j}}0^{t_j}$ and the previous word is
$$w^{(j+1)}=w_1^{(j)}\cdots w_{a_j}^{(j)}\epsilon_1\cdots\epsilon_{n_{i_j}-1}(\epsilon_{n_{i_j}}-1)\epsilon_1\cdots\epsilon_{t_j}=:w_1^{(j+1)}\cdots w_{a_{j+1}}^{(j+1)}\epsilon_1\cdots\epsilon_{t_j},$$
where $a_{j+1}=a_j+n_{i_j}$. By Lemma \ref{tailnotfull}, $w^{(j+1)}$ is also not full. At this time, $w^{(j+1)},w^{(j)},\cdots,w^{(2)},w^{(1)}$ are $j+1$ consecutive non-full words in $\Sigma_\beta^n$.
\newline$\cdots$
\newline Noting that $n_1=1$, it is obvious that there exist $d\in\N$ such that $w^{(d)},\cdots,w^{(1)}$ are not full, and $s=n_{i_1}+n_{i_2}+\cdots+n_{i_d}$, i.e., $n_{i_d}=t_{d-1}$. Then $\epsilon_{t_{d-1}}>0$ and $w^{(d+1)}=w_1^{(d)}\cdots w_{a_d}^{(d)}\epsilon_1\cdots\epsilon_{t_{d-1}-1}(\epsilon_{t_{d-1}}-1)$ is full by Lemma \ref{admfull}. We get the conclusion since $\tau_\beta(s)=d$.
\end{proof}


\begin{corollary}[The maximal length of the runs of non-full words]\label{maxlengthnotfull}Let $\beta>1$ with $\beta\notin\N$. Then
$$\max N_\beta^n=\left\{\begin{array}{ll}
\max\{\tau_\beta(s):1\le s\le n\} & \mbox{if } \epsilon(1,\beta) \mbox{ is infinite;}\\
\max\{\tau_\beta(s):1\le s\le min\{M-1,n\}\} & \mbox{if } \epsilon(1,\beta) \mbox{ is finite with length } M.
\end{array}\right.$$
\end{corollary}
\begin{proof}Let $l\in N_\beta^n$ and $[w^{(l)},w^{(l-1)},\cdots,w^{(2)},w^{(1)}]\in\mathcal{N}_\beta^n$. Then, by Corollary \ref{tail-non-full}, there exists
$$\left\{\begin{array}{ll}
1\le s_0\le n & \mbox{if } \epsilon(1,\beta) \mbox{ is infinite}\\
1\le s_0\le \min\{M-1,n\} & \mbox{if } \epsilon(1,\beta) \mbox{ is finite with length } M
\end{array}\right.$$
such that $w^{(1)}=w_1^{(1)}\cdots w_{n-s_0}^{(1)}\epsilon_1\cdots\epsilon_{s_0}$ and we have $l=\tau_\beta(s_0)$ by Lemma \ref{tau-non-full}. Therefore
$$\max N_\beta^n\le\left\{\begin{array}{ll}
\max\{\tau_\beta(s):1\le s\le n\} & \mbox{if } \epsilon(1,\beta) \mbox{ is infinite}\\
\max\{\tau_\beta(s):1\le s\le min\{M-1,n\}\} & \mbox{if } \epsilon(1,\beta) \mbox{ is finite with length } M
\end{array}\right.$$
by the randomicity of the selection of $l$. On the other hand, the equality follows from the fact that $0^{n-t_0}\epsilon_1\cdots\epsilon_{t_0}\in\Sigma_\beta^n$ included, the previous consecutive $\tau_\beta(t_0)$ words are not full by Lemma \ref{tau-non-full} where
$$\tau_\beta(t_0)=\left\{\begin{array}{ll}
\max\{\tau_\beta(s):1\le s\le n\} & \mbox{if } \epsilon(1,\beta) \mbox{ is infinite;}\\
\max\{\tau_\beta(s):1\le s\le min\{M-1,n\}\} & \mbox{if } \epsilon(1,\beta) \mbox{ is finite with length } M.
\end{array}\right.$$
\end{proof}

In the following we give an example to show how to calculate the maximal length of the runs of non-full words in $\Sigma_\beta^n$.

\begin{example}Let $n=8$ and $\epsilon(1,\beta)=\epsilon_{n_1}0\epsilon_{n_2}000\epsilon_{n_3}0\cdots 0\epsilon_{n_4}0\cdots 0\epsilon_{n_5}0\cdots$,
where $n_1=1,n_2=3,n_3=7,n_4>8,\epsilon_{n_i}\neq 0$ for any $i$. Then, by Corollary \ref{maxlengthnotfull}, the maximal length of the runs of non-full words in $\Sigma_\beta^8$ is $\max\{\tau_\beta(s):1\le s\le 8\}$. Since \newline$\begin{array}{llllllll}
1=1 & \Rightarrow\tau_\beta(1)=1; && 2=1+1 & \Rightarrow\tau_\beta(2)=2; && 3=3 & \Rightarrow\tau_\beta(3)=1;\\
4=3+1 & \Rightarrow\tau_\beta(4)=2; && 5=3+1+1 & \Rightarrow\tau_\beta(5)=3; && 6=3+3 & \Rightarrow\tau_\beta(6)=2;\\
7=7 & \Rightarrow\tau_\beta(7)=1; && 8=7+1 & \Rightarrow\tau_\beta(8)=2,
\end{array}$
\newline we get that $\max\{\tau_\beta(s):1\le s\le 8\}=3$ is the maximal length.
\end{example}

\begin{corollary}\label{notfulllengthupperbound}Let $\beta>1$. We have $\max N_\beta^n\le r_n(\beta)+1$ for any $n\in\N$. Moreover, if $\epsilon(1,\beta)$ is finite with length $M$, then $\max N_\beta^n\le r_{M-1}(\beta)+1$ for any $n\in\N$.
\end{corollary}
\begin{proof}If $\epsilon(1,\beta)$ is infinite, then
$$\max N_\beta^n=\max\{\tau_\beta(s):1\le s\le n\}\le\max\{r_s(\beta)+1:1\le s\le n\}=r_n(\beta)+1.$$
If $\epsilon(1,\beta)$ is finite with length $M$, then
$$\max N_\beta^n=\max\{\tau_\beta(s):1\le s\le\min\{M-1,n\}\}\le\max\{r_s(\beta)+1:1\le s\le\min\{M-1,n\}\}.$$
and we have $\max N_\beta^n\le r_n(\beta)+1$ and $\max N_\beta^n\le r_{M-1}(\beta)+1$.
\end{proof}

\begin{remark}\label{n+1}Combining Corollary \ref{maxlengthnotfull} and $\tau_\beta(n)\le n$ (or Corollary \ref{notfulllengthupperbound} and $r_n(\beta)+1\le n$), we have $\max N_\beta^n\le n$ for any $n\in\N$ which contains the result about the distribution of full cylinders given by Bugeaud and Wang \cite[Theorem 1.2]{BuWa14}. Moreover, if $\epsilon(1,\beta)$ is finite with length $M$, then $\max N_\beta^n\le M-1$ for any $n\in\N$. If $\beta\in A_0$ which is a class of $\beta$ given by Li and Wu \cite{LiWu08}, then $\max N_\beta^n$ has the upper bound $\max\limits_{s\ge1}r_s(\beta)+1$ which does not rely on $n$.
\end{remark}

\begin{theorem}[The lengths of the maximal runs of non-full words]\label{notfulllength}Let $\beta>1$ with $\beta\notin\N$ and $\{n_i\}$ be the nonzero sequence of $\beta$. Then $N_\beta^n$ is given by the following table.
$$\begin{tabular}{|c|c|c|c|c|c|}
\hline
\multicolumn{4}{|c|}{Condition} & Conclusion & \multirow{2}{*}{Case} \\
\cline{1-5}
$\beta$ & \multicolumn{3}{|c|}{$\epsilon(1,\beta)$} & $N_\beta^n=$ & \\
\hline
\hline
\multirow{2}{*}{$\beta>2$} & \multicolumn{3}{|c|}{infinite} & $D_1$ & (1) \\
\cline{2-6}
& \multicolumn{3}{|c|}{finite with length $M$} & $D_2$ & (2) \\
\cline{1-6}
\multirow{8}{*}{$1<\beta<2$} & \multirow{2}{*}{infinite} & \multicolumn{2}{|c|}{$n<n_2$} & $\{n\}$ & (3) \\
\cline{3-6}
&& \multicolumn{2}{|c|}{$n\ge n_2$} & $D_5$ & (4) \\
\cline{2-6}
& \multirow{6}{*}{finite with length $M$} & \multirow{3}{*}{$n_2=M$} & $n<M$ & $\{n\}$ & (5) \\
\cline{4-6}
&&& $n=M$ & $\{M-1\}$ & (6) \\
\cline{4-6}
&&& $n>M$ & $D_4$ & (7) \\
\cline{3-6}
&& \multirow{3}{*}{$n_2<M$} & $n<n_2$ & $\{n\}$ & (8) \\
\cline{4-6}
&&& $n_2\le n < M$ & $D_5$ & (9) \\
\cline{4-6}
&&& $n\ge M$ & $D_3$ & (10) \\
\hline
\end{tabular}$$
$\begin{array}{rl}
\mbox{Here } & D_1=\{1,2,\cdots,\max\{\tau_\beta(s):1\le s\le n\}\}; \\
& D_2=\{1,2,\cdots,\max\{\tau_\beta(s):1\le s\le \min\{M-1,n\}\}\}; \\
& D_3=\{1,2,\cdots,\max\{\tau_\beta(s):1\le s\le M-1\}\}; \\
& D_4=\{1,2,\cdots,\min\{n-M,M-1\}\}\cup\{M-1\}; \\
& D_5=\{1,2,\cdots,\min\{n_2-1,n-n_2+1\}\}\cup\{\tau_\beta(s):n_2-1\le s\le n\}.
\end{array}$
\end{theorem}

\begin{corollary}[The minimal length of the maximal runs of non-full words]\label{minlengthnotfull}Let $\beta>1$ with $\beta\notin\N$ and $\{n_i\}$ be the nonzero sequence of $\beta$. Then
$$\min N_\beta^n=\left\{\begin{array}{ll}
M-1 & \mbox{if } 1<\beta<2 \mbox{ and } \epsilon(1,\beta) \mbox{ is finite with length } M=n_2=n; \\
n   & \mbox{if } 1<\beta<2 \mbox{ and } n<n_2; \\
1   & \mbox{otherwise.}
\end{array}\right.$$
\end{corollary}
\begin{proof}It follows from Theorem \ref{notfulllength}.
\end{proof}

\begin{proof}[Proof of Theorem \ref{notfulllength}] We prove the conclusions for the cases (1)-(10) from simple ones to complicate as below.

Cases (3), (5) and (8) can be proved together. When $1<\beta<2$ and $n<n_2$, no matter $\epsilon(1,\beta)$ is finite or not, noting that $\lfloor\beta\rfloor=1$ and $\epsilon(1,\beta)|_{n_2}=10^{n_2-2}1$, we get $\epsilon_1\cdots\epsilon_n=10^{n-1}$. Then all the elements in $\Sigma_\beta^n$ from small to large are $0^n$, $0^{n-1}1$, $0^{n-2}10$, $\cdots$, $10^{n-1}$, where $0^n$ is full and the others are all not full by Lemma \ref{tailnotfull}. Therefore $N_\beta^n=\{n\}$.

Case (6). When $1<\beta<2$, $\epsilon(1,\beta)$ is finite with length $M$ and $n=n_2=M$, noting that $\lfloor\beta\rfloor=1$ and $\epsilon(1,\beta)=10^{M-2}10^\infty$, all the elements in $\Sigma_\beta^n$ from small to large are $0^M$, $0^{M-1}1$, $0^{M-2}10$, $\cdots$, $010^{M-2}$, $10^{M-1}$, where $0^M$ is full, $10^{M-1}$ is also full by Proposition \ref{exaDIV} (2) and the others are all not full by Lemma \ref{tailnotfull}. Therefore $N_\beta^n=\{M-1\}$.

Case (1). When $\beta>2$ and $\epsilon(1,\beta)$ is infinite, it suffices to prove $N_\beta^n\supset D_1$ since the reverse inclusion follows immediately from Corollary \ref{maxlengthnotfull}. By Proposition \ref{pro} (4), it suffices to show $N_\beta^n\supset\{\tau_\beta(s):1\le s\le n\}$. In fact:
\begin{itemize}
\item[\textcircled{\small{$1$}}] For any $1\le s\le n-1$, let $u=0^{n-s-1}10^s$. It is full by $\epsilon_1=\lfloor\beta\rfloor\ge 2$ and Corollary \ref{tail-non-full}. The previous word $u^{(1)}=0^{n-s}\epsilon_1\cdots\epsilon_s$ is not full by Lemma \ref{tailnotfull}. So $\tau_\beta(s)\in N_\beta^n$ by Lemma \ref{tau-non-full}.
\item[\textcircled{\small{$2$}}] For $s=n$, combining the fact that $\epsilon_1\cdots\epsilon_s$ is maximal in $\Sigma_\beta^n$ and Lemma \ref{tau-non-full}, we get $\tau_\beta(s)\in N_\beta^n$.
\end{itemize}
Therefore $N_\beta^n=D_1$.

Case (2) can be proved by similar way as Case (1).

Case (10). When $1<\beta<2$, $\epsilon(1,\beta)$ is finite with length $M$ and $n_2<M\le n$, we have $\epsilon(1,\beta)=10^{n_2-2}1\epsilon_{n_2+1}\cdots\epsilon_M0^\infty$. It suffices to prove $N_\beta^n\supset D_3$ since the reverse inclusion follows immediately from Corollary \ref{maxlengthnotfull}. By Proposition \ref{pro} (4), it suffices to show $N_\beta^n\supset\{\tau_\beta(s):1\le s\le M-1\}$. In fact:
\begin{itemize}
\item[\textcircled{\small{$1$}}] For any $n_2-1\le s\le M-1$, let $u=0^{n-s-1}10^s$. It is full by $s\ge n_2-1$ and Corollary \ref{tail-non-full}. The previous word $u^{(1)}=0^{n-s}\epsilon^*_1\cdots\epsilon^*_s=0^{n-s}\epsilon_1\cdots\epsilon_s$ is not full by Lemma \ref{tailnotfull}. So $\tau_\beta(s)\in N_\beta^n$ by Lemma \ref{tau-non-full}.
\item[\textcircled{\small{$2$}}] For any $1\le s\le n_2-2$, we get $n_2-1\le n_3-n_2$ by Lemma \ref{1exp}. So $1\le s\le n_2-2\le n_3-n_2-1\le M-n_2-1\le n-n_2-1$ and then $n-n_2-s\ge 1$. Let
$$u=0^{n-n_2-s}10^{n_2+s-1}.$$
It is full by $n_2+s-1\ge n_2-1$ and Corollary \ref{tail-non-full}. Noting that $n_2\le n_2+s-1<n_3$, the previous word of $u$ is
$$\begin{array}{rl}
u^{(1)}&=0^{n-n_2-s+1}\epsilon^*_1\cdots\epsilon^*_{n_2+s-1}\\
&=0^{n-n_2-s+1}\epsilon_1\cdots\epsilon_{n_2+s-1}\\
&=0^{n-n_2-s+1}10^{n_2-2}10^{s-1}\\
&=0^{n-n_2-s+1}10^{n_2-2}\epsilon_1\cdots\epsilon_s
\end{array}$$
which is not full by Lemma \ref{tailnotfull}. So $\tau_\beta(s)\in N_\beta^n$ by Lemma \ref{tau-non-full}.
\end{itemize}
Therefore $N_\beta^n=D_3$.

Case (7). When $1<\beta<2$, $\epsilon(1,\beta)$ is finite with length $M$ and $n>n_2=M$, we have $\epsilon(1,\beta)=10^{M-2}10^\infty$.
\newline On the one hand, we prove $N_\beta^n\subset D_4$. Let $l\in N_\beta^n$ and $[w^{(l)},w^{(l-1)},\cdots,w^{(2)},w^{(1)}]\in\mathcal{N}_\beta^n$. By Corollary \ref{tail-non-full}, there exist $1\le s\le M-1$,  $2\le n-M+1\le a\le n-1$ such that $a+s=n$ and $w^{(1)}=w_1\cdots w_a\epsilon_1\cdots\epsilon_s$.
Then $l=\tau_\beta(s)=s$ by Lemma \ref{tau-non-full} and $s\le n_2-1$. Moreover, $w^{(1)}=w_1\cdots w_a10^{s-1}$.
\begin{itemize}
\item[\textcircled{\small{$1$}}] If $w_1\cdots w_a=0^a$, then the next word of $w^{(1)}$ is $w:=0^{a-1}10^s$ which is full by $[w^{(l)}$, $w^{(l-1)}$, $\cdots$, $w^{(2)}$, $w^{(1)}]$ $\in\mathcal{N}_\beta^n$. Combining $s\le M-1$ and Corollary \ref{tail-non-full}, we get $s=M-1$. Hence $l=M-1\in D_4$.
\item[\textcircled{\small{$2$}}] If $w_1\cdots w_a\neq0^a$, we get $a\ge M$ by $w_{k+1}\cdots w_a10^\infty\prec\epsilon(1,\beta)=10^{M-2}10^\infty$ for any $k\ge0$. Hence $s\le n-M$ and $l=s\in D_4$.
\end{itemize}
On the other hand, we prove $N_\beta^n\supset D_4$.
\begin{itemize}
\item[\textcircled{\small{$1$}}] For $M-1$, let $u=0^{n-M}10^{M-1}$ which is full by Corollary \ref{tail-non-full}. The consecutive previous words are $u^{(1)}=0^{n-M+1}10^{M-2},\cdots,u^{(M-1)}=0^{n-1}1,u^{(M)}=0^n$ where $u^{(1)},\cdots,u^{(M-1)}$ are not full by Lemma \ref{tailnotfull}, and $u^{(M)}$ is full. Therefore $M-1\in N_\beta^n$.
\item[\textcircled{\small{$2$}}] For any $1\le s\le\min\{n-M,M-1\}$, let
$$u^{(1)}=0^{n-M-s}\epsilon^*_1\cdots\epsilon^*_{M+s}=0^{n-M-s}10^{M-1}10^{s-1}=0^{n-M-s}10^{M-1}\epsilon_1\cdots\epsilon_s.$$
i) If $s=n-M$, then $u^{(1)}=\epsilon_1^*\cdots\epsilon_{M+s}^*$ is maximal in $\Sigma_\beta^n$.
\newline ii) If $s<n-M$, i.e.,$n-M-s-1\ge0$, then the next word of $u^{(1)}$ is $0^{n-M-s-1}10^{M+s}$ which is full by Corollary \ref{tail-non-full}.
\newline Hence we must have $s=\tau_\beta(s)\in N_\beta^n$ by $s\le n_2-1$ and Lemma \ref{tau-non-full}.
\end{itemize}
Therefore $N_\beta^n=D_4$.

Cases (4) and (9) can be proved together. When $1<\beta<2$, $\epsilon(1,\beta)$ is infinite with $n\ge n_2$ or $\epsilon(1,\beta)$ is finite with length $M$ and $n_2\le n<M$, we have $\epsilon(1,\beta)=10^{n_2-2}1\epsilon_{n_2+1}\epsilon_{n_2+2}\cdots$. By Proposition \ref{pro} (2), we get
$$D_5=\{\tau_\beta(s):1\le s\le\min\{n_2-1,n-n_2+1\}\text{ or }n_2-1\le s\le n\}.$$
On the one hand, we prove $N_\beta^n\subset D_5$. Let $l\in N_\beta^n$ and $[w^{(l)},w^{(l-1)},\cdots,w^{(2)},w^{(1)}]\in\mathcal{N}_\beta^n$. By Corollary \ref{tail-non-full}, there exist $1\le s\le n$, $0\le a\le n-1$ such that $a+s=n$ and $w^{(1)}=w_1\cdots w_a\epsilon_1\cdots\epsilon_s$. Then $l=\tau_\beta(s)$ by Lemma \ref{tau-non-full}.
\begin{itemize}
\item[\textcircled{\small{$1$}}] If $a=0$, then $s=n$ and $l=\tau_\beta(n)\in D_5$.
\item[\textcircled{\small{$2$}}] If $a\ge 1$, we divide it into two cases.
\newline i) If $w_1\cdots w_a=0^a$, then the next word of $w^{(1)}$ is $0^{a-1}10^s$ which is full by $[w^{(l)}$, $w^{(l-1)}$, $\cdots$, $w^{(2)}$, $w^{(1)}]$ $\in\mathcal{N}_\beta^n$. Combining $\epsilon(1,\beta)=10^{n_2-2}1\epsilon_{n_2+1}\epsilon_{n_2+2}\cdots$ and Corollary \ref{tail-non-full}, we get $s\ge n_2-1$. Hence $l=\tau_\beta(s)\in D_5$.
\newline ii) If $w_1\cdots w_a\neq0^a$, we get $a\ge n_2-1$ by $w_{k+1}\cdots w_a10^\infty\prec\epsilon(1,\beta)=10^{n_2-2}1\epsilon_{n_2+1}\epsilon_{n_2+2}\cdots$ for any $k\ge0$. Hence $s\le n-n_2+1$.
\newline\textcircled{\small{$a$}} If $s\ge n_2-1$, then $l=\tau_\beta(s)\in\{\tau_\beta(s):n_2-1\le s\le n\}\subset D_5$.
\newline\textcircled{\small{$b$}} If $s\le n_2-1$, then $l=\tau_\beta(s)\in\{\tau_\beta(s):1\le s\le\min\{n_2-1,n-n_2+1\}\}\subset D_5$.
\end{itemize}
On the other hand, we prove $N_\beta^n\supset D_5$.
\begin{itemize}
\item[\textcircled{\small{$1$}}] For any $n_2-1\le s\le n$, let $u^{(1)}=0^{n-s}\epsilon^*_1\cdots\epsilon^*_s$. No matter whether $\epsilon(1,\beta)$ is infinite or finite with length $M>n$ (which implies $s<M$), we get $u^{(1)}=0^{n-s}\epsilon_1\cdots\epsilon_s$ which is not full by Lemma \ref{tailnotfull}.
\newline i) If $s=n$, then $u^{(1)}=\epsilon^*_1\cdots\epsilon^*_n$ is maximal in $\Sigma_\beta^n$.
\newline ii) If $n_2-1\le s\le n-1$, then the next word of $u^{(1)}$ is $0^{n-s-1}10^s$ which is full by $s\ge n_2-1$ and Corollary \ref{tail-non-full}.
\newline Hence we must have $\tau_\beta(s)\in N_\beta^n$ by Lemma \ref{tau-non-full}.
\item[\textcircled{\small{$2$}}] For any $1\le s\le \min\{n_2-1,n-n_2+1\}$, let
$$u^{(1)}=0^{n-n_2-s+1}\epsilon^*_1\cdots\epsilon^*_{n_2+s-1}.$$
No matter $\epsilon(1,\beta)$ is infinite or finite with length $M>n$ (which implies $n_2+s-1\le n<M$), we get
$$u^{(1)}=0^{n-n_2-s+1}\epsilon_1\cdots\epsilon_{n_2+s-1}.$$
Since Lemma \ref{1exp} implies $n_2-1\le n_3-n_2$, we get $1\le s\le n_2-1\le n_3-n_2$ and then $n_2\le n_2+s-1<n_3$. Hence
$$u^{(1)}=0^{n-n_2-s+1}10^{n_2-2}10^{s-1}$$
$$\indent\indent=0^{n-n_2-s+1}10^{n_2-2}\epsilon_1\cdots\epsilon_s$$
which is not full by Lemma \ref{tailnotfull}.
\newline i) If $s=n-n_2+1$, then $u^{(1)}=\epsilon^*_1\cdots\epsilon^*_n$ is maximal in $\Sigma_\beta^n$.
\newline ii) If $s<n-n_2+1$, i.e., $n-n_2-s\ge 0$, then the next word of $u^{(1)}$ is $0^{n-n_2-s}10^{n_2+s-1}$ which is full by Corollary \ref{tail-non-full}.
\newline Hence we must have $\tau_\beta(s)\in N_\beta^n$ by Lemma \ref{tau-non-full}.
\end{itemize}
Therefore $N_\beta^n=D_5$.
\end{proof}

\begin{remark}It follows from Theorem \ref{notfulllength} that the lengths of the maximal runs of non-full words rely on the positions of nonzero terms in $\epsilon(1,\beta)$, i.e., $\{n_i\}$.
\end{remark}

\begin{ack}
The work was supported by NSFC 11671151 and Guangdong Natural Science Foundation 2014A030313230.
\end{ack}

\end{document}